\documentclass{amsart}
\usepackage{amssymb}
\usepackage{amsfonts}

\newtheorem{theorem}{Theorem}
\theoremstyle{plain}

\newtheorem{corollary}{Corollary}

\newtheorem{definition}{Definition}
\newtheorem{example}{Example}

\newtheorem{proposition}{Proposition}
\newtheorem{remark}{Remark}

\numberwithin{equation}{section}
\input{tcilatex}

\begin{document}
\title[On $(h-s)_{1,2}-$convex Functions and Hadamard-type Inequalities]{On $%
(h-s)_{1,2}-$convex Functions and Hadamard-type Inequalities}
\author{$^{\bigstar }$M. EM\.{I}N \"{O}ZDEM\.{I}R}
\address{$^{\bigstar }$Atat\"{u}rk University, K. K. Education Faculty,
Department of Mathematics, 25240, Campus, Erzurum, Turkey }
\email{emos@atauni.edu.tr}
\author{$^{\blacksquare }$Mevl\"{u}t TUN\c{C}}
\address{$^{\blacksquare }$Kilis 7 Aral\i k University, Faculty of Science
and Arts, Department of Mathematics, 79000, Kilis, Turkey}
\email{mevlutunc@kilis.edu.tr}
\author{$^{\spadesuit }$Ahmet Ocak AKDEM\.{I}R}
\address{$^{\spadesuit }$A\u{g}r\i\ \.{I}brahim \c{C}e\c{c}en University
Faculty of Science and Letters, Department of Mathematics, 04100, A\u{g}r\i
, Turkey}
\email{ahmetakdemir@agri.edu.tr}
\keywords{$h-$convex, $s-$convex, Bullen's inequality. \ \ \ \ \ \ \ \ \ \ \
\ \ \ \ \ \ \ \ \ \ \ \ \ \ \ \ \ \ \ \ \ \ \ \ \ \ \ \ \ \ \ \ \ \ \ \ \ \
\ \ \ \ \ \ \ \ \ \ \ \ \ \ \ \ \ \ \ \ \ \ \ \ \ \ \ \ $^{\blacksquare
}:Corresponding$ $Author.$}

\begin{abstract}
In this paper, two new classes of convex functions as a generalization of
convexity which is called $(h-s)_{1,2}-$convex functions are given. We also
prove some \ Hadamard-type inequalities and applications to the special
means are given.
\end{abstract}

\maketitle

\section{Introduction}

The following definition is well known in the literature [\ref{mit2}]: A
function $f:I\rightarrow 
%TCIMACRO{\U{211d} }%
%BeginExpansion
\mathbb{R}
%EndExpansion
,\emptyset \neq I\subseteq 
%TCIMACRO{\U{211d} }%
%BeginExpansion
\mathbb{R}
%EndExpansion
,$ is said to be convex on $I$ if inequality

\begin{equation}
f\left( tx+\left( 1-t\right) y\right) \leq tf\left( x\right) +\left(
1-t\right) f\left( y\right)  \label{101}
\end{equation}%
holds for all $x,y\in I$ and $t\in \left[ 0,1\right] $. Geometrically, this
means that if $P,Q$ and $R$ are three distinct points on the graph of $f$
with $Q$ between $P$ and $R$, then $Q$ is on or below chord $PR$.

Let $f:I\subseteq 
%TCIMACRO{\U{211d} }%
%BeginExpansion
\mathbb{R}
%EndExpansion
\rightarrow 
%TCIMACRO{\U{211d} }%
%BeginExpansion
\mathbb{R}
%EndExpansion
$ be a convex function and $a,b\in I$ with $a<b$. The following double
inequality:$\ $

\begin{equation}
f\left( \frac{a+b}{2}\right) \leq \frac{1}{b-a}\int_{a}^{b}f\left( x\right)
dx\leq \frac{f\left( a\right) +f\left( b\right) }{2}  \label{102}
\end{equation}%
is known in the literature as Hadamard's inequality (or Hermite-Hadamard
inequality) for convex functions. Keep in mind that some of the classical
inequalities for means can come from (\ref{102}) for convenient particular
selections of the function $f.$ If $f$ is concave, this double inequality
hold in the inversed way.

\begin{remark}
\cite{SEL} Note that the first inequality stronger than the second
inequality in (\ref{102}); i.e., the following inequality is valid for a
convex function $f:$%
\begin{equation}
\frac{1}{b-a}\int_{a}^{b}f\left( x\right) dx-f\left( \frac{a+b}{2}\right)
\leq \frac{f\left( a\right) +f\left( b\right) }{2}-\frac{1}{b-a}%
\int_{a}^{b}f\left( x\right) dx.  \label{k}
\end{equation}%
Indeed (\ref{k}) can be written as 
\begin{equation}
\frac{2}{b-a}\int_{a}^{b}f\left( x\right) dx\leq \frac{1}{2}\left[ f\left(
a\right) +f\left( b\right) +2f\left( \frac{a+b}{2}\right) \right] ,
\label{z}
\end{equation}%
which is 
\begin{eqnarray*}
&&\frac{2}{b-a}\int_{a}^{\frac{a+b}{2}}f\left( x\right) dx+\frac{2}{b-a}%
\int_{\frac{a+b}{2}}^{b}f\left( x\right) dx \\
&\leq &\frac{1}{2}\left[ f\left( a\right) +f\left( \frac{a+b}{2}\right) %
\right] +\frac{1}{2}\left[ f\left( b\right) +f\left( \frac{a+b}{2}\right) %
\right] .
\end{eqnarray*}%
this immediately follows by applying the second inequality in (\ref{102})
twice (on the interval $\left[ a,\frac{a+b}{2}\right] $ and $\left[ \frac{a+b%
}{2},b\right] $). By letting $a=-1,$ $b=1,$ we obtain the result due to
Bullen (1978). Further on, we shall call (\ref{k}) as Bullen's inequality.
\end{remark}

The inequalities (\ref{102}) which have numerous uses in a variety of
settings, has been came a significant groundwork in mathematical analysis
and optimization. Many reports have provided new proof, extensions and
considering its refinements, generalizations, numerous interpolations and
applications, for example, in the theory of special means and information
theory. For some results on generalizations, extensions and applications of
the Hermite-Hadamard inequalities and convexity, see \cite{SSD1}-\cite{SEL}.

\begin{definition}
\bigskip \textit{[\ref{god}] We say that }$f:I\rightarrow 
%TCIMACRO{\U{211d} }%
%BeginExpansion
\mathbb{R}
%EndExpansion
$\textit{\ is Godunova-Levin function or that }$f$\textit{\ belongs to the
class }$Q\left( I\right) $\textit{\ if }$f$\textit{\ is non-negative and for
all }$x,y\in I$\textit{\ and }$t\in \left( 0,1\right) $\textit{\ we have \ \
\ \ \ \ \ \ \ \ \ \ \ }%
\begin{equation}
f\left( tx+\left( 1-t\right) y\right) \leq \frac{f\left( x\right) }{t}+\frac{%
f\left( y\right) }{1-t}.  \label{103}
\end{equation}
\end{definition}

\begin{definition}
\textit{[\ref{dr1}] We say that }$f:I\subseteq 
%TCIMACRO{\U{211d} }%
%BeginExpansion
\mathbb{R}
%EndExpansion
\rightarrow 
%TCIMACRO{\U{211d} }%
%BeginExpansion
\mathbb{R}
%EndExpansion
$\textit{\ is a }$P-$\textit{function or that }$f$\textit{\ belongs to the
class }$P\left( I\right) $\textit{\ if }$f$\textit{\ is nonnegative and for
all }$x,y\in I$\textit{\ and }$t\in \left[ 0,1\right] ,$\textit{\ we have}%
\begin{equation}
f\left( tx+\left( 1-t\right) y\right) \leq f\left( x\right) +f\left(
y\right) .  \label{104}
\end{equation}
\end{definition}

\begin{definition}
\textit{[\ref{hud}] Let }$s\in \left( 0,1\right] .$\textit{\ A function }$%
f:\left( 0,\infty \right] \rightarrow \left[ 0,\infty \right] $\textit{\ is
said to be }$s-$\textit{convex in the second sense if \ \ \ \ \ \ \ \ \ \ \
\ }%
\begin{equation}
f\left( tx+\left( 1-t\right) y\right) \leq t^{s}f\left( x\right) +\left(
1-t\right) ^{s}f\left( y\right) ,  \label{105}
\end{equation}%
\textit{for all }$x,y\in \left( 0,b\right] $\textit{\ \ and }$t\in \left[ 0,1%
\right] $\textit{. This class of }$s-$convex\textit{\ functions is usually
denoted by }$K_{s}^{2}$\textit{.}
\end{definition}

In 1978, Breckner introduced $s-$convex functions as a generalization of
convex functions in [\textit{\ref{bre1}}]. Also, in that work Breckner
proved the important fact that the set valued map is $s-$convex only if the
associated support function is $s-$convex function in [\textit{\ref{bre2}}].
A number of properties and connections with s-convex in the first sense are
discussed in paper [\textit{\ref{hud}}]. Of course, $s-$convexity means just
convexity when $s=1$.

\begin{definition}
\bigskip \textit{[\ref{var}] Let }$h:J\subseteq 
%TCIMACRO{\U{211d} }%
%BeginExpansion
\mathbb{R}
%EndExpansion
\rightarrow 
%TCIMACRO{\U{211d} }%
%BeginExpansion
\mathbb{R}
%EndExpansion
$\textit{\ be a positive function . We say that }$f:I\subseteq 
%TCIMACRO{\U{211d} }%
%BeginExpansion
\mathbb{R}
%EndExpansion
\rightarrow 
%TCIMACRO{\U{211d} }%
%BeginExpansion
\mathbb{R}
%EndExpansion
$\textit{\ is }$h-$\textit{convex function, or that }$f$\textit{\ belongs to
the class }$SX\left( h,I\right) $\textit{, if }$f$\textit{\ is nonnegative
and for all }$x,y\in I$\textit{\ and }$t\in \left[ 0,1\right] $\textit{\ we
have \ \ \ \ \ \ \ \ \ \ \ \ }%
\begin{equation}
f\left( tx+\left( 1-t\right) y\right) \leq h\left( t\right) f\left( x\right)
+h\left( 1-t\right) f\left( y\right) .  \label{106}
\end{equation}
\end{definition}

If inequality (\ref{106}) is reversed, then $f$ is said to be $h-$concave,
i.e. $f\in SV\left( h,I\right) $. Obviously, if $h\left( t\right) =t$, then
all nonnegative convex functions belong to $SX\left( h,I\right) $\ and all
nonnegative concave functions belong to $SV\left( h,I\right) $; if $h\left(
t\right) =\frac{1}{t}$, then $SX\left( h,I\right) =Q\left( I\right) $; if $%
h\left( t\right) =1$, then $SX\left( h,I\right) \supseteq P\left( I\right) $%
; and if $h\left( t\right) =t^{s}$, where $s\in \left( 0,1\right) $, then $%
SX\left( h,I\right) \supseteq K_{s}^{2}$.

In [\textit{\ref{dr1}}], Dragomir \textit{et al.} proved two inequalities of
Hadamard-type for $P-$functions.

\begin{theorem}
\lbrack \textit{\ref{dr1}] Let }$f\in P\left( I\right) $\textit{, }$a,b\in I$%
\textit{, with }$a<b$\textit{\ and }$f\in L_{1}\left( \left[ a,b\right]
\right) $\textit{. Then} \ \ \ \ \ \ \ \ \ \ \ 
\begin{equation}
\ f\left( \frac{a+b}{2}\right) \leq \frac{2}{b-a}\int_{a}^{b}f\left(
x\right) dx\leq 2\left[ f\left( a\right) +f\left( b\right) \right] .
\label{107}
\end{equation}
\end{theorem}

In [\ref{pachpatte}], Pachpatte established the new following Hadamard-type
inequality for products of convex functions.

\begin{theorem}
\lbrack \ref{pachpatte}] \textit{Let }$f,g:\left[ a,b\right] \rightarrow %
\left[ 0,\infty \right) $\textit{\ \ be convex functions on }$\left[ a,b%
\right] \subset 
%TCIMACRO{\U{211d} }%
%BeginExpansion
\mathbb{R}
%EndExpansion
$\textit{, }$a<b$\textit{. Then } \ \ \ \ \ \ \ \ \ \ \ 
\begin{equation}
\frac{1}{b-a}\int_{a}^{b}f\left( x\right) g\left( x\right) dx\leq \frac{1}{3}%
M\left( a,b\right) +\frac{1}{6}N\left( a,b\right)  \label{108}
\end{equation}%
\textit{where }$M\left( a,b\right) =f\left( a\right) g\left( a\right)
+f\left( b\right) g\left( b\right) $\textit{\ and }$N\left( a,b\right)
=f\left( a\right) g\left( b\right) +f\left( b\right) g\left( a\right) $%
\textit{.}\bigskip
\end{theorem}

In [\ref{ssd6}], Dragomir and Fitzpatrick proved a new variety of Hadamard's
inequality which holds for $s-$convex functions in the second sense.\bigskip

\begin{theorem}
\lbrack \ref{ssd6}] \textit{Suppose that }$f:\left[ 0,\infty \right)
\rightarrow \left[ 0,\infty \right) $\textit{\ is an }$s-$\textit{convex
function in the second sense, where }$s\in \left( 0,1\right) $\textit{, and
let }$a,b\in \left[ 0,\infty \right) ,a<b.$\textit{\ If }$f\in L_{1}\left( %
\left[ a,b\right] \right) $\textit{, then the following inequalities hold:}%
\begin{equation}
2^{s-1}f\left( \frac{a+b}{2}\right) \leq \frac{1}{b-a}\int_{a}^{b}f\left(
x\right) dx\leq \frac{f\left( a\right) +f\left( b\right) }{s+1}.  \label{109}
\end{equation}
\end{theorem}

Up until now, there are many reports on convexity and Hadamard-type
inequalities. The main purpose of the present paper is to give new classes
of convex functions which called $\left( h-s\right) _{1,2}-$convex functions
as a generalization of ordinary convex functions and to prove new
Hadamard-type inequalities for these new classes of functions. Some
applications to the special meansof real numbers are given. Throughout this
paper we will imply $M(a,b)=f(a)g(a)+f(b)g(b)$ and $%
N(a,b)=f(a)g(b)+f(b)g(a). $\bigskip

\section{New Definitions and Results}

\begin{definition}
Let $h:J\subset 
%TCIMACRO{\U{211d} }%
%BeginExpansion
\mathbb{R}
%EndExpansion
\rightarrow 
%TCIMACRO{\U{211d} }%
%BeginExpansion
\mathbb{R}
%EndExpansion
$ \ be a non-negative function, $h\neq 0.$We say that $f:%
%TCIMACRO{\U{211d} }%
%BeginExpansion
\mathbb{R}
%EndExpansion
^{+}\cup \left\{ 0\right\} \rightarrow 
%TCIMACRO{\U{211d} }%
%BeginExpansion
\mathbb{R}
%EndExpansion
$ is an $\left( h-s\right) _{1}-$convex function in the first sense, or that 
$f$ belong to the class $SX(\left( h-s\right) _{1},I),$ if $f$ is
non-negative and for all $x,y\in \left[ 0,\infty \right) =I,$ $s\in \left(
0,1\right] ,$ $t\in \left[ 0,1\right] $ we have%
\begin{equation}
f(tx+(1-t)y)\leq h^{s}(t)f(x)+(1-h^{s}(t))f(y).  \label{21}
\end{equation}%
\bigskip If inequality (\ref{21}) is reversed, then $f$ is said to be $%
\left( h-s\right) _{1}-$concav function in the first sense, i.e., $f\in
SV(\left( h-s\right) _{1},I).$\bigskip
\end{definition}

\begin{definition}
Let $\ \ \ \ h:J\subset 
%TCIMACRO{\U{211d} }%
%BeginExpansion
\mathbb{R}
%EndExpansion
\rightarrow 
%TCIMACRO{\U{211d} }%
%BeginExpansion
\mathbb{R}
%EndExpansion
$ be a non-negative function,\ $h\neq 0.$ We say that $f:%
%TCIMACRO{\U{211d} }%
%BeginExpansion
\mathbb{R}
%EndExpansion
^{+}\cup \left\{ 0\right\} \rightarrow 
%TCIMACRO{\U{211d} }%
%BeginExpansion
\mathbb{R}
%EndExpansion
$ \ \ is an \ $(h-s)_{2}-$convex function in the second sense, or that $f$
belong to the class $SX(\left( h-s\right) _{2},I)$ , if $f$ is non-negative
and for all $u,v\in \left[ 0,\infty \right) =I,$ $s\in \left( 0,1\right] ,$ $%
t\in \left[ 0,1\right] $ we have%
\begin{equation}
f(tu+(1-t)v)\leq h^{s}(t)f(u)+h^{s}(1-t)f(v).  \label{22}
\end{equation}%
If inequality (\ref{22}) is reversed, then $f$ is said to be $(h-s)_{2}-$%
concav function in the second sense, i.e., $f\in SV(\left( h-s\right)
_{2},I).$\bigskip
\end{definition}

Obviously, in (\ref{22}), if $h(t)=t$, then all $s-$convex functions in the
second sense belongs to $SX(\left( h-s\right) _{2},I)$ and all $s-$concav
functions in the second sense belongs to $SV(\left( h-s\right) _{2},I)$, and
it can be easily seen that for $h(t)=t,$ $s=1,$ $(h-s)_{2}-$convexity
reduces to ordinary convexity defined on $\left[ 0,\infty \right) .$
Similarly, in (\ref{21}), if $h(t)=t$, then all $s-$convex functions in the
first sense belongs to $SX(\left( h-s\right) _{1},I)$ and all $s-$concav
functions in the first sense belongs to $SV(\left( h-s\right) _{1},I)$, and
it can be easily seen that for $h(t)=t$ $,s=1,$ $(h-s)_{1}-$convexity
reduces to ordinary convexity defined on $\left[ 0,\infty \right) .$

\begin{example}
Let $h(t)=t$ be a function and let the function $f$ be defined as following;%
\begin{equation*}
f:[2,4]\rightarrow 
%TCIMACRO{\U{211d} }%
%BeginExpansion
\mathbb{R}
%EndExpansion
^{+},\text{ \ \ \ \ }f(x)=\ln x.
\end{equation*}%
Then $f$ is non-convex and non-$h-$convex function, but it is $(h-s)_{2}-$%
convex function.
\end{example}

The following theorem was obtained by using the $(h-s)_{2}-$convex function
in the second sense.

\begin{theorem}
\label{th1} Let $h:J\subset 
%TCIMACRO{\U{211d} }%
%BeginExpansion
\mathbb{R}
%EndExpansion
\rightarrow 
%TCIMACRO{\U{211d} }%
%BeginExpansion
\mathbb{R}
%EndExpansion
$ be a non-negative function, $h\neq 0.$ We say that $f:I=\left[ 0,\infty
\right) \rightarrow 
%TCIMACRO{\U{211d} }%
%BeginExpansion
\mathbb{R}
%EndExpansion
$ is an $(h-s)_{2}-$convex function in the second sense, or that $f$ belong
to the class $SX(\left( h-s\right) _{2},I),$ if $f$ is non-negative and for
all $x,y\in \left[ 0,\infty \right) =I,$ $s\in \left( 0,1\right] ,$ $t\in %
\left[ 0,1\right] .$ If $f\in L_{1}\left[ a,b\right] ,$ $h\in L_{1}\left[ 0,1%
\right] $, we have the following inequality:%
\begin{equation}
\frac{1}{b-a}\int_{a}^{b}f(x)dx\leq f(a)\int_{0}^{1}h^{s}\left( t\right)
dt+f(b)\int_{0}^{1}h^{s}\left( 1-t\right) dt.  \label{23}
\end{equation}

\begin{proof}
By the definition of $(h-s)_{2}-$convex mappings in the second sense, for
any $s\in \left( 0,1\right] $ and $t\in \left[ 0,1\right] ,$ we obtain the
following inequality for $u=a,$ $y=b$%
\begin{equation}
f(ta+(1-t)b)\leq h^{s}(t)f(a)+h^{s}(1-t)f(b).  \label{24}
\end{equation}%
Integrating both side of (\ref{24}) with respect to $t$ on $\left[ 0,1\right]
$, we have 
\begin{equation*}
\int_{0}^{1}f(ta+(1-t)b)dt\leq f(a)\int_{0}^{1}h^{s}\left( t\right)
dt+f(b)\int_{0}^{1}h^{s}\left( 1-t\right) dt.
\end{equation*}%
Use of the changing variable $ta+(1-t)b=x,$ $(b-a)dt=dx$, we have%
\begin{equation*}
\frac{1}{b-a}\int_{a}^{b}f(x)dx\leq f(a)\int_{0}^{1}h^{s}\left( t\right)
dt+f(b)\int_{0}^{1}h^{s}\left( 1-t\right) dt
\end{equation*}%
which is the inequality in (\ref{23}).
\end{proof}
\end{theorem}

\begin{corollary}
In the inequality (\ref{23}); \textit{if we choose }$s=1$\textit{, we have
the inequality;}%
\begin{equation*}
\frac{1}{b-a}\int_{a}^{b}f(x)dx\leq \left[ f(a)+f(b)\right]
\int_{0}^{1}h\left( t\right) dt
\end{equation*}%
\textit{\ }
\end{corollary}

\begin{remark}
\textit{If we choose }$h(t)=t$\textit{, we have the inequality;}%
\begin{eqnarray*}
\frac{1}{b-a}\int_{a}^{b}f(x)dx &\leq
&f(a)\int_{0}^{1}t^{s}dt+f(b)\int_{0}^{1}\left( 1-t\right) ^{s}dt \\
&=&\frac{f(a)+f(b)}{s+1}
\end{eqnarray*}%
which is the right hand side of the inequality in (\ref{109}).
\end{remark}

\begin{remark}
In the inequality in (\ref{23}); If we choose, $h(t)=t$ and $s=1$, we have%
\begin{equation*}
\frac{1}{b-a}\int_{a}^{b}f(x)dx\leq
f(a)\int_{0}^{1}tdt+f(b)\int_{0}^{1}\left( 1-t\right) dt=\frac{f(a)+f(b)}{2}
\end{equation*}%
which is the right hand side of the Hermite-Hadamard inequality in (\ref{102}%
).
\end{remark}

\begin{theorem}
\label{th2} Let $h:J\subset 
%TCIMACRO{\U{211d} }%
%BeginExpansion
\mathbb{R}
%EndExpansion
\rightarrow 
%TCIMACRO{\U{211d} }%
%BeginExpansion
\mathbb{R}
%EndExpansion
$ be a non-negative function, $h\neq 0.$ We say that $f:I=\left[ 0,\infty
\right) \rightarrow 
%TCIMACRO{\U{211d} }%
%BeginExpansion
\mathbb{R}
%EndExpansion
$ is an $(h-s)_{2}-$convex function in the second sense, or that $f$ belong
to the class $SX(\left( h-s\right) _{2},I),$ if $f$ is non-negative and for
all $x,y\in \left[ 0,\infty \right) =I,$ $s\in \left( 0,1\right] ,$ $t\in %
\left[ 0,1\right] .$ If $f\in L_{1}\left[ a,b\right] ,$ $h\in L_{1}\left[ 0,1%
\right] $, we have the following inequality:%
\begin{equation}
\frac{1}{2h^{s}\left( \frac{1}{2}\right) }f\left( \frac{a+b}{2}\right) \leq 
\frac{1}{b-a}\int_{a}^{b}f(x)dx\leq \frac{f(a)+f(b)}{2}\int_{0}^{1}\left[
h^{s}\left( t\right) +h^{s}\left( 1-t\right) \right] dt.  \label{25}
\end{equation}
\end{theorem}

\begin{proof}
\bigskip By the $(h-s)_{2}-$convexity of $f,$ we have that%
\begin{equation*}
f\left( \frac{x+y}{2}\right) \leq h^{s}\left( \frac{1}{2}\right) f\left(
x\right) +h^{s}\left( \frac{1}{2}\right) f\left( y\right) .
\end{equation*}%
If we choose $x=ta+(1-t)b,$ $y=tb+(1-t)a$, we get%
\begin{equation}
f\left( \frac{a+b}{2}\right) \leq h^{s}\left( \frac{1}{2}\right) f\left(
ta+(1-t)b\right) +h^{s}\left( \frac{1}{2}\right) f\left( tb+(1-t)a\right)
\label{26}
\end{equation}%
for all $t\in \left[ 0,1\right] $. Then, integrating both side of (\ref{26})
with respect to $t$ on $\left[ 0,1\right] ,$ we have%
\begin{equation*}
f\left( \frac{a+b}{2}\right) \leq \int_{0}^{1}\left( h^{s}\left( \frac{1}{2}%
\right) f\left( ta+(1-t)b\right) +h^{s}\left( \frac{1}{2}\right) f\left(
tb+(1-t)a\right) \right) dt.
\end{equation*}%
Use of the changing of variable, we have%
\begin{equation}
\frac{1}{2h^{s}\left( \frac{1}{2}\right) }f\left( \frac{a+b}{2}\right) \leq 
\frac{1}{b-a}\int_{a}^{b}f(x)dx,  \label{a}
\end{equation}%
which is the first inequality in $\left( \ref{25}\right) .$

To prove the second inequality in $\left( \ref{25}\right) $, we use the
right side of (\ref{26}) and using $(h-s)_{2}-$convexity of $f$ , we have%
\begin{eqnarray*}
&&h^{s}\left( \frac{1}{2}\right) \left[ f\left( ta+(1-t)b\right) +f\left(
tb+(1-t)a\right) \right] \\
&\leq &h^{s}\left( \frac{1}{2}\right) \left[ h^{s}\left( t\right) f\left(
a\right) +h^{s}\left( 1-t\right) f\left( b\right) +h^{s}\left( t\right)
f\left( b\right) +h^{s}\left( 1-t\right) f\left( a\right) \right] \\
&=&h^{s}\left( \frac{1}{2}\right) \left[ h^{s}\left( t\right) +h^{s}\left(
1-t\right) \right] \left[ f\left( a\right) +f\left( b\right) \right]
\end{eqnarray*}%
Integrating the both side of the above inequality,we have%
\begin{eqnarray}
&&h^{s}\left( \frac{1}{2}\right) \int_{0}^{1}\left[ f\left( ta+(1-t)b\right)
+f\left( tb+(1-t)a\right) \right] dt  \notag \\
&=&h^{s}\left( \frac{1}{2}\right) \frac{2}{b-a}\int_{a}^{b}f\left( x\right)
dx  \notag \\
&\leq &h^{s}\left( \frac{1}{2}\right) \left[ f(a)+f(b)\right] \int_{0}^{1}%
\left[ h^{s}\left( t\right) +h^{s}\left( 1-t\right) \right] dt.  \label{28}
\end{eqnarray}%
We obtain the inequality in (\ref{25}).
\end{proof}

\begin{remark}
In the inequality (\ref{25}); \textit{if we choose }$h(t)=t$\textit{, we
have the inequality}%
\begin{equation*}
2^{s-1}f\left( \frac{a+b}{2}\right) \leq \frac{1}{b-a}\int_{a}^{b}f(x)dx\leq 
\frac{f(a)+f(b)}{s+1}
\end{equation*}%
which is the inequality (\ref{109}).
\end{remark}

\begin{remark}
\textit{If we choose }$h(t)=t$ and $s=1$\textit{, we have the inequality}%
\begin{equation*}
f\left( \frac{a+b}{2}\right) \leq \frac{1}{b-a}\int_{a}^{b}f(x)dx\leq \frac{%
f(a)+f(b)}{2}
\end{equation*}%
which is the Hermite-Hadamard inequality.
\end{remark}

\begin{theorem}
\label{th3} Let $h:J\subset 
%TCIMACRO{\U{211d} }%
%BeginExpansion
\mathbb{R}
%EndExpansion
\rightarrow 
%TCIMACRO{\U{211d} }%
%BeginExpansion
\mathbb{R}
%EndExpansion
$ be a non-negative function, $h\neq 0.$ We say that $f,g:I=\left[ 0,\infty
\right) \rightarrow 
%TCIMACRO{\U{211d} }%
%BeginExpansion
\mathbb{R}
%EndExpansion
$ are an $(h-s)_{2}-$convex function in the second sense$,$ if $f,g$ are
non-negative and for all $x,y\in \left[ 0,\infty \right) =I,$ $s\in \left(
0,1\right] ,$ $t\in \left[ 0,1\right] .$ If $fg\in L_{1}\left[ a,b\right] ,$ 
$h\in L_{1}\left[ 0,1\right] $, we have the following inequality;%
\begin{eqnarray}
\frac{1}{b-a}\int_{a}^{b}f(x)g(x)dx &\leq &f\left( a\right) g\left( a\right)
\int_{0}^{1}h^{2s}\left( t\right) dt+f\left( b\right) g\left( b\right)
\int_{0}^{1}h^{2s}\left( 1-t\right) dt  \notag \\
&&+\left[ f\left( a\right) g\left( b\right) +f\left( b\right) g\left(
a\right) \right] \int_{0}^{1}h^{s}\left( t\right) h^{s}\left( 1-t\right) dt.
\label{29}
\end{eqnarray}
\end{theorem}

\begin{proof}
Since $f,g\in SX(h-s)_{2},$ we have%
\begin{eqnarray*}
f\left( ta+\left( 1-t\right) b\right) &\leq &h^{s}\left( t\right) f\left(
a\right) +h^{s}\left( 1-t\right) f\left( b\right) \\
g\left( ta+(1-t)b\right) &\leq &h^{s}\left( t\right) g\left( a\right)
+h^{s}\left( 1-t\right) g\left( b\right)
\end{eqnarray*}%
for all $s\in \left( 0,1\right] ,$ $t\in \left[ 0,1\right] .$ Since $f$ and $%
g$ are non-negative,%
\begin{eqnarray*}
&&f\left( ta+\left( 1-t\right) b\right) g\left( ta+(1-t)b\right) \\
&\leq &\left[ h^{s}\left( t\right) f\left( a\right) +h^{s}\left( 1-t\right)
f\left( b\right) \right] \left[ h^{s}\left( t\right) g\left( a\right)
+h^{s}\left( 1-t\right) g\left( b\right) \right] \\
&=&h^{2s}\left( t\right) f\left( a\right) g\left( a\right) +h^{s}\left(
t\right) h^{s}\left( 1-t\right) f\left( a\right) g\left( b\right) \\
&&+h^{s}\left( t\right) h^{s}\left( 1-t\right) f\left( b\right) g\left(
a\right) +h^{2s}\left( 1-t\right) f\left( b\right) g\left( b\right) .
\end{eqnarray*}%
Then if we integrate the both side of the above inequality with respect to $%
t $ on $\left[ 0,1\right] ,$ we have the inequality in (\ref{29}).
\end{proof}

In the next corollary we will also make use of the Beta function of Euler
type, which is for $x,y>0$ defined as%
\begin{equation*}
\beta \left( x,y\right) =\int_{0}^{1}t^{x-1}\left( 1-t\right) ^{y-1}dt=\frac{%
\Gamma \left( x\right) \Gamma \left( y\right) }{\Gamma \left( x+y\right) }.
\end{equation*}

\begin{corollary}
\bigskip In the inequality (\ref{29}), \textit{if we choose }$h(t)=t$ and $%
s\in \left( 0,1\right] $\textit{, we have}%
\begin{eqnarray*}
\frac{1}{b-a}\int_{a}^{b}\left( fg\right) (x)dx &\leq &f\left( a\right)
g\left( a\right) \int_{0}^{1}t^{2s}dt+f\left( b\right) g\left( b\right)
\int_{0}^{1}\left( 1-t\right) ^{2s}dt \\
&&+\left[ f\left( a\right) g\left( b\right) +f\left( b\right) g\left(
a\right) \right] \int_{0}^{1}t^{s}\left( 1-t\right) ^{s}dt \\
&=&\frac{M\left( a,b\right) }{2s+1}+N\left( a,b\right) \beta \left(
s+1,s+1\right) \\
&=&\frac{M\left( a,b\right) }{2s+1}+N\left( a,b\right) \frac{\Gamma
^{2}\left( s+1\right) }{\Gamma \left( 2s+2\right) }.
\end{eqnarray*}
\end{corollary}

\begin{remark}
\bigskip In the inequality (\ref{29}), \textit{if we choose }$h(t)=t$ and $%
s=1$\textit{, we have}%
\begin{eqnarray*}
\frac{1}{b-a}\int_{a}^{b}\left( fg\right) (x)dx &\leq &f\left( a\right)
g\left( a\right) \int_{0}^{1}t^{2}dt+f\left( b\right) g\left( b\right)
\int_{0}^{1}\left( 1-t\right) ^{2}dt \\
&&+\left[ f\left( a\right) g\left( b\right) +f\left( b\right) g\left(
a\right) \right] \int_{0}^{1}t\left( 1-t\right) dt \\
&=&\frac{M\left( a,b\right) }{3}+\frac{N\left( a,b\right) }{6}
\end{eqnarray*}%
which is the inequality in (\ref{108}).
\end{remark}

$\bigskip $

\begin{theorem}
\label{th4} Let $h:J\subset 
%TCIMACRO{\U{211d} }%
%BeginExpansion
\mathbb{R}
%EndExpansion
\rightarrow 
%TCIMACRO{\U{211d} }%
%BeginExpansion
\mathbb{R}
%EndExpansion
$ be a non-negative function, $h\neq 0.$ We say that $f:I=\left[ 0,\infty
\right) \rightarrow 
%TCIMACRO{\U{211d} }%
%BeginExpansion
\mathbb{R}
%EndExpansion
$ is an $(h-s)_{2}-$convex function in the second sense$,$ if $f$ is
non-negative and for all $x,y\in \left[ 0,\infty \right) =I,$ $s\in \left(
0,1\right] ,$ $t\in \left[ 0,1\right] .$ If $f\in L_{1}\left[ a,b\right] ,$ $%
h\in L_{1}\left[ 0,1\right] $, we have 
\begin{equation}
\frac{1}{b-a}\int_{a}^{b}f(x)dx\leq \left[ f\left( a\right) +f\left(
b\right) \right] \int_{0}^{1}h^{s}\left( t\right) dt.  \label{30}
\end{equation}
\end{theorem}

\begin{proof}
\bigskip Since $f\in SX(h-s)_{2},$ by the definition $(h-s_{2})$ convexity
of $f$ , we can write%
\begin{eqnarray*}
f\left( ta+\left( 1-t\right) b\right) &\leq &h^{s}\left( t\right) f\left(
a\right) +h^{s}\left( 1-t\right) f\left( b\right) \\
f\left( tb+\left( 1-t\right) a\right) &\leq &h^{s}\left( t\right) f\left(
b\right) +h^{s}\left( 1-t\right) f\left( a\right)
\end{eqnarray*}%
for all $s\in \left( 0,1\right] ,$ $t\in \left[ 0,1\right] .$ If we add the
above inequalities, we write%
\begin{equation*}
f\left( ta+\left( 1-t\right) b\right) +f\left( tb+\left( 1-t\right) a\right)
\leq \left[ f\left( a\right) +f\left( b\right) \right] \left[ h^{s}\left(
t\right) +h^{s}\left( 1-t\right) \right]
\end{equation*}%
Integrating the both side of the above inequality with respect to $t$ on $%
\left[ 0,1\right] ,$ we have%
\begin{equation*}
\int_{0}^{1}f\left( ta+\left( 1-t\right) b\right) +f\left( tb+\left(
1-t\right) a\right) dt\leq \left[ f\left( a\right) +f\left( b\right) \right]
\int_{0}^{1}\left[ h^{s}\left( t\right) +h^{s}\left( 1-t\right) \right] dt
\end{equation*}%
By use of the changing of variable and taking into account the $%
\int_{0}^{1}h^{s}\left( t\right) dt=\int_{0}^{1}h^{s}\left( 1-t\right) dt$
for $s\in \left( 0,1\right] ,$ we get%
\begin{equation*}
\frac{2}{b-a}\int_{a}^{b}f(x)dx\leq 2\left[ f\left( a\right) +f\left(
b\right) \right] \int_{0}^{1}h^{s}\left( t\right) dt
\end{equation*}%
which completes the proof.
\end{proof}

\begin{remark}
If in (\ref{30}), we choose $h(t)=t,$ we have the right hand side of the
inequality (\ref{109}). Again, if in (\ref{30}), we choose $h(t)=t$ and $%
s=1, $ we have the right hand side of Hermite-Hadamard inequality. Again, if
in (\ref{30}), we choose $h(t)=1,$ we have the right hand side of
Hermite-Hadamard inequality for $P-$convex functions in (\ref{107}).
\end{remark}

\begin{theorem}
\label{th5} Let $h:J\subset 
%TCIMACRO{\U{211d} }%
%BeginExpansion
\mathbb{R}
%EndExpansion
\rightarrow 
%TCIMACRO{\U{211d} }%
%BeginExpansion
\mathbb{R}
%EndExpansion
$ be a non-negative function, $h\neq 0.$ We say that $f:I=\left[ 0,\infty
\right) \rightarrow 
%TCIMACRO{\U{211d} }%
%BeginExpansion
\mathbb{R}
%EndExpansion
$ is an $(h-s)_{2}-$convex function in the second sense$,$ if $f$ is
non-negative and for all $x,y\in \left[ 0,\infty \right) =I,$ $s\in \left(
0,1\right] ,$ $t\in \left[ 0,1\right] .$ If $f\in L_{1}\left[ a,b\right] ,$ $%
h\in L_{1}\left[ 0,1\right] $, we have 
\begin{equation}
\frac{1}{b-a}\int_{a}^{b}f(x)dx\leq \left[ \frac{f\left( a\right) +f\left(
b\right) }{2}+f\left( \frac{a+b}{2}\right) \right] \int_{0}^{1}h^{s}\left(
t\right) dt.  \label{31}
\end{equation}
\end{theorem}

\begin{proof}
\bigskip By the $(h-s)_{2}-$convexity of $f$ , we have%
\begin{eqnarray*}
f\left( ta+\left( 1-t\right) \frac{a+b}{2}\right) &\leq &h^{s}\left(
t\right) f\left( a\right) +h^{s}\left( 1-t\right) f\left( \frac{a+b}{2}%
\right) \\
f\left( t\frac{a+b}{2}+\left( 1-t\right) b\right) &\leq &h^{s}\left(
t\right) f\left( \frac{a+b}{2}\right) +h^{s}\left( 1-t\right) f\left(
b\right)
\end{eqnarray*}%
If we integrate the both side of the above inequalities with respect to $t$
on $[0,1]$, and use of the changing of variable, we get%
\begin{equation*}
\frac{2}{b-a}\int_{a}^{\frac{a+b}{2}}f\left( x\right) dx\leq f\left(
a\right) \int_{0}^{1}h^{s}\left( t\right) dt+f\left( \frac{a+b}{2}\right)
\int_{0}^{1}h^{s}\left( 1-t\right) dt
\end{equation*}%
and%
\begin{equation*}
\frac{2}{b-a}\int_{\frac{a+b}{2}}^{b}f\left( x\right) dx\leq f\left( \frac{%
a+b}{2}\right) \int_{0}^{1}h^{s}\left( t\right)
dt+f(b)\int_{0}^{1}h^{s}\left( 1-t\right) dt.
\end{equation*}%
By adding the above inequalities and taking into account the $%
\int_{0}^{1}h^{s}\left( t\right) dt=\int_{0}^{1}h^{s}\left( 1-t\right) dt$
for $s\in \left( 0,1\right] ,$ we get%
\begin{equation*}
\frac{2}{b-a}\int_{a}^{b}f(x)dx\leq \left[ f\left( a\right) +f\left(
b\right) +2f\left( \frac{a+b}{2}\right) \right] \int_{0}^{1}h^{s}\left(
t\right) dt
\end{equation*}%
which completes the proof.
\end{proof}

\begin{corollary}
\bigskip If in (\ref{31}), we choose $h(t)=t,$ we have%
\begin{equation*}
\frac{1}{b-a}\int_{a}^{b}f(x)dx\leq \frac{1}{s+1}\left[ \frac{f\left(
a\right) +f\left( b\right) }{2}+f\left( \frac{a+b}{2}\right) \right] .
\end{equation*}%
\bigskip
\end{corollary}

\begin{remark}
If in (\ref{31}), we choose $h(t)=t$ and $s=1,$ we have%
\begin{equation*}
\frac{1}{b-a}\int_{a}^{b}f(x)dx\leq \frac{1}{2}\left[ \frac{f\left( a\right)
+f\left( b\right) }{2}+f\left( \frac{a+b}{2}\right) \right]
\end{equation*}%
which is the inequality (\ref{z}).
\end{remark}

\section{\protect\LARGE Applications to Some Special Means}

We now consider the applications of our Theorems to the following special
means

b) The arithmetic mean:%
\begin{equation*}
A=A\left( a,b\right) :=\frac{a+b}{2},\text{\ \ }a,b\geq 0,
\end{equation*}

c) The geometric mean: 
\begin{equation*}
G=G\left( a,b\right) :=\sqrt{ab},\text{ \ }a,b\geq 0,
\end{equation*}

d) The harmonic mean:%
\begin{equation*}
H=H\left( a,b\right) :=\frac{2ab}{a+b},\text{ \ }a,b\geq 0,
\end{equation*}

e) The quadratic mean:%
\begin{equation*}
K=K\left( a,b\right) :=\sqrt{\frac{a^{2}+b^{2}}{2}}\text{ \ }a,b\geq 0,
\end{equation*}

f) The logarithmic mean:

\begin{equation*}
L=L\left( a,b\right) :=\left\{ 
\begin{array}{l}
a\text{ \ \ \ \ \ \ \ \ \ \ \ \ \ if \ \ }a=b \\ 
\frac{b-a}{\ln b-\ln a}\text{ \ \ \ \ \ if \ \ }a\neq b%
\end{array}%
\right. ,\text{ \ }a,b\geq 0,
\end{equation*}

g) The Identric mean.

\begin{equation*}
I=I\left( a,b\right) :=\left\{ 
\begin{array}{l}
a\text{ \ \ \ \ \ \ \ \ \ \ \ \ \ \ \ \ \ if \ \ }a=b \\ 
\frac{1}{e}\left( \frac{b^{b}}{a^{a}}\right) ^{\frac{1}{b-a}}\text{ \ \ \ \
\ if \ \ }a\neq b%
\end{array}%
\right. ,\text{ \ }a,b\geq 0,
\end{equation*}

h) The $p-$logarithmic mean:

\begin{equation*}
L_{p}=L_{p}\left( a,b\right) :=\left\{ 
\begin{array}{l}
\left[ \frac{b^{p+1}-a^{p+1}}{\left( p+1\right) \left( b-a\right) }\right]
^{1/p}\text{ \ \ \ \ \ if \ \ }a\neq b \\ 
a\text{ \ \ \ \ \ \ \ \ \ \ \ \ \ \ \ \ \ \ \ \ \ \ \ \ if \ \ }a=b%
\end{array}%
\right. ,\text{ \ }p\in 
%TCIMACRO{\U{211d} }%
%BeginExpansion
\mathbb{R}
%EndExpansion
\backslash \left\{ -1,0\right\} ;\text{ \ }a,b>0.
\end{equation*}

\bigskip

\bigskip The following inequality is well known in the literature:%
\begin{equation*}
H\leq G\leq L\leq I\leq A\leq K
\end{equation*}

It is also known that $L_{p}$ is monotonically increasing over $p\in 
%TCIMACRO{\U{211d} }%
%BeginExpansion
\mathbb{R}
%EndExpansion
,$ denoting $L_{0}=I$ and $L_{-1}=L.$

The following propositions holds:

\begin{proposition}
\bigskip Let $a,b\in \left( 2,\infty \right) $, $a<b.$ Then for all $s\in %
\left[ 0,1\right] ,$ we have%
\begin{equation}
\ln I\left( a,b\right) \leq \frac{2}{\left( s+1\right) }A\left( \ln a,\ln
b\right)  \tag{3.2}
\end{equation}
\end{proposition}

\begin{proof}
The proof is obvious from Theorem 4 applied $f\left( x\right) =\ln x$ $,$ $%
h\left( t\right) =t,$ $x\in \left[ 2,\infty \right) $ and $s\in \left[ 0,1%
\right] .$
\end{proof}

\begin{proposition}
Let $a,b\in \left( 2,\infty \right) .$ Then for all $s\in \left[ 0,1\right]
, $ we have%
\begin{equation}
2^{s-1}\ln \left( A(a,b)\right) \leq \ln I\left( a,b\right) \leq \frac{1}{%
\left( s+1\right) }A\left( \ln a,\ln b\right)  \tag{3.1}
\end{equation}
\end{proposition}

\begin{proof}
The assertion follows from Theorem 5 applied for $f\left( x\right) =\ln x,$ $%
x\in \left[ 2,\infty \right) $ and $h\left( t\right) =t.$
\end{proof}

\begin{proposition}
Let $a,b\in \left( 2,\infty \right) .$ Then for all $s\in \left[ 0,1\right]
, $ we have%
\begin{equation}
\ln I\left( a,b\right) \leq \frac{1}{\left( s+1\right) }\ln G^{2}\left(
a,b\right)  \tag{3.3}
\end{equation}
\end{proposition}

\begin{proof}
The prof is immediate follows from Theorem 7 applied for $f\left( x\right)
=\ln x,$ $x\in \left[ 2,\infty \right) $ and $h\left( t\right) =t.$
\end{proof}

\begin{proposition}
Let $a,b\in \left( 2,\infty \right) .$ Then for all $s\in \left[ 0,1\right]
, $ we have%
\begin{equation}
\ln I\left( a,b\right) \leq \frac{1}{\left( s+1\right) }\left[ \frac{\ln
G^{2}\left( a,b\right) }{2}+\ln \left( A(a,b)\right) \right]  \tag{3.3}
\end{equation}
\end{proposition}

\begin{proof}
The prof is immediate follows from Theorem 8 applied for $f\left( x\right)
=\ln x,$ $x\in \left[ 2,\infty \right) $ and $h\left( t\right) =t.$
\end{proof}

\end{document}